\documentclass{article}
\usepackage{latexsym}
\usepackage{mathrsfs}
\usepackage{indentfirst}
\usepackage{amsmath,amsthm, amssymb}
\usepackage[dvips]{graphicx}
\usepackage{graphicx}
\usepackage{epsfig}
\usepackage{verbatim}
\usepackage{color}

\usepackage{caption}

\usepackage{enumitem}

\usepackage{epstopdf}

\usepackage{thmtools}

\newtheorem{theorem}{Theorem}[section]
\newtheorem{lemma}[theorem]{Lemma}

\newtheorem{conjecture}{Conjecture}

\begin{document}
\title{Cycles of lengths $3$ and $n-1$ in digraphs under a Bang-Jensen-Gutin-Li type conditon
}
\author{
Zan-Bo Zhang$^{1,2}$\thanks{Email address: eltonzhang2001@gmail.com, zanbozhang@gdufe.edu.cn. Supported by
Guangdong Basic and Applied Basic Research Foundation (2024A1515012286).},
Wenhao Wu$^1$,
Weihua He$^3$\thanks{ Corresponding author.
E-mail address: hwh12@gdut.edu.cn. Supported by National Natural Science Foundation of China (Grant Number: 12426663)}
\\ \small $^1$ School of Statistics and Mathematics, Guangdong University of
\\ \small  Finance and Economics, Guangzhou 510320, China
\\ \small $^2$ Institute of Artificial Intelligence and Deep Learning, Guangdong University
\\ \small of Finance and Economics, Guangzhou 510320, China
\\ \small $^3$ School of Mathematics and Statistics, Guangdong University
\\ \small of Technology, Guangzhou 510006, China
}
\date{}
\maketitle

\begin{abstract}
\noindent
Bang-Jensen-Gutin-Li type conditions are the conditions for hamiltonicity of digraphs
which impose degree restrictions on nonadjacent vertices which have a common in-neighbor
or a common out-neighbor.
They can be viewed as an extension of Fan type conditions in undirected graphs, as well as
generalization of locally (in-, out-)semicomplete digraphs.
Since their first appearance in 1996, various Bang-Jensen-Gutin-Li type conditions for hamitonicity have come forth.
In this paper we establish a condition of Bang-Jensen-Gutin-Li type which implies not only a hamiltonian cycle
but also a $3$-cycle and an $(n-1)$-cycle, with well-characterized exceptional graphs.
We conjecture that this condition implies the existence of cycle of every length.
$\newline$\noindent\textbf{Key words}: Hamiltonian cycle, pancyclicity, Bang-Jensen-Gutin-Li type condition,
Locally semicomplete digraph

\end{abstract}

\section{Introduction and terminology} \label{sec:Intro}

We consider finite digraphs without loops and multiple arcs. Let $D$ be a digraph,
and $u,v\in V(D)$. If there is an arc from $u$ to $v$ we write $u\rightarrow v$,
otherwise $u \nrightarrow v$.
Let $F$ and $H$ be two disjoint subdigraphs of $D$, or two disjoint subset of $V(D)$,
if every vertex in $F$ sends an arc to every vertex in $H$ we write $F\rightarrow H$.
By $d(F,H)$ we denote the number of arcs between $F$ and $H$.
Let $u\in V(D)$ and $F$ a subdigraph of $D$ (it is possible that $u\in V(F)$),
by $d_F(u)$ we mean the number of arcs between $u$ and $V(F)$.
Let $C$ be a cycle and $u,v\in V(C)$ where $u\neq v$,
by $C[u,v]$ we mean the segment of $C$ from $u$ to $v$.
For a path $P$ and $u,v\in V(P)$ where $u\neq v$,
$P[u,v]$ means the segment of $P$ from $u$ to $v$.
A $C$-\emph{bypass} is a $(u,v)$-path $P$ with length at least two,
and endvertices $u$, $v$ but no other vertex on $C$,
where $u\neq v$. When $C$ is clear from the context, we only call $P$ a bypass.
The path $C[u,v]$ is the \emph{gap} of $P$.
The \emph{girth} of $D$ is the length of the shortest cycle in $D$, denoted by $g(D)$.
In an undirected graph, the distance between $u$ and $v$ is denoted by $d(u,v)$.
For terminology not defined here please be referred to \cite{BJG2008}.

Degree conditions are among the most basic conditions for the existence of hamiltonian cycles in graphs.
For digraphs, the most popular hamiltonian degree conditions are possibly
Ghouila-Houri condition and Meyniel condition (see Chapter 6.4 of \cite{BJG2008}), which involve
degree of one vertex and degree sum of two nonadjacent vertices respectively.
In 1996, Bang-Jensen, Gutin and Li (\cite{BJGL1996}) introduced conditions that only impose degree restrictions
on nonadjacent vertices which have a common in-neighbor or a common out-neighbor.
For a pair of nonadjacent vertices $x$ and $y$ in a digraph $D$,
if they have a common in-neighbor (resp. out-neighbor),
we say that they form a nonadjacent dominated pair (resp. a nonadjacent dominating pair).
One of the main theorems of Bang-Jensen, Gutin and Li has the following form.

\begin{theorem} \label{thm:BGLhamilton}
(\cite{BJGL1996})
Let $D$ be a strong digraph on $n\ge 2$ verties. Suppose that,
for every nonadjacent dominated pair $\{x,y\}$,
either $d(x) \ge n$ and $d(y) \ge n - 1$ or $d(x) \ge n - 1$ and $d(y) \ge n$. Then $D$ is hamiltonian.
\end{theorem}

The conditions involving nonadjacent dominated pairs and nonadjacent dominating pairs
can be viewed as the extension of the following Fan condition in undirected graphs,
for nonadjacent dominated pairs and nonadjacent dominating pairs in a digraph $D$ are of distance $2$
in the underlying graph of $D$.

\begin{theorem} (\cite{Fan1984})
Let $G$ be a $2$-connected graph on $n \ge 3$ vertices and let $u$ and $v$
be distinct vertices of $G$. If $d(u, v) = 2$ implies that $\max(d(u), d(v)) \ge n/2$,
then $G$ has a hamiltonian cycle.
\end{theorem}

Such conditions are also generalization of three kinds of well studied digraphs, namely,
\emph{locally semicomplete digraphs} (LSD for short),
\emph{locally in-semicomplete digraphs }(LiSD for short)
and \emph{locally out-semicomplete digraphs} (LoSD for short).
An LiSD (resp. LoSD) is a digraph in which the in-neighborhood (resp. out-neighborhood)
of each vertex induces a semicomplete digraph. An LSD is a digraph which is both an LiSD and an LoSD.
An LSD without $2$-cycle is called a \emph{local tournament}.
The hamiltonicity of LiSD, LoSD and LSD has all been verified (\cite{BJHP1993}).

The meta-conjecture of Bondy (\cite{Bondy1973}) that any nontrivial condition for hamiltonicity also
implies pancyclicity up to a small class of exceptional graphs had motivated
the studied on \emph{pancyclicity}, which means the existence of cycle of every length from
$3$ to $|G|$ in a graph $G$. In view of this idea, various Fan type conditions
has been proved to imply pancyclicity.
On the other hand, the pancyclicity of LSD has also been verified (\cite{Guo1995}),
the detail of which is shown below and will be used in our work.

A digraph on $n$ vertices is called a \emph{round digraph} if we can label its
vertices $v_0$, $v_1, \dots , v_{n-1}$ such that for each $i$,
$N^+(v_i) = \{v_{i+1}, \dots , V_{i+d^+(v_i)} \}$ and
$N^-(v_i) = \{v_{i-d^-(v)}, \dots , v_{i-1}\}$, where the subscripts are taken modulo $n$.
An LSD $D$ is said to be \emph{round decomposable}, if it can be represented in the form
$D = R[S_1, S_2, \dots , S_r]$, where $R$ is a local tournament with $r$ vertices
$\{v_1, v_2, \dots, v_r\}$, and $D$ is obtained from $R$ by replacing every vertex
$v_i$ with a strong semicomplete digraph $S_i$, every arc $v_iv_j$
with all arcs from $V(S_i)$ to $V(S_j)$.
$D = R[S_1, S_2, \dots , S_r]$ is called the \emph{round decomposition} of $D$.

\begin{theorem} \label{thm:nonpancyclicLSD} (\cite{Guo1995})
A strong LSD $D$ is pancyclic if and only if it is not of the form
$D = R[S_1, S_2, \dots , S_r]$, where $R$ is a round local tournament with
$g(R) > max\{2, |V(S_1)|,|V(S_2)|, ... ,|V(S_r)|\} +1$.
\end{theorem}

It then sounds natural to ask for a pancyclic condition of Bang-Jensen-Gutin-Li type.
Some attempts have been made. For example,
Darbinyan and Karapetyan (\cite{DK2017}) have proved the existence of cycle of length $n-1$,
under the condition of Theorem \ref{thm:BGLhamilton} with an additional semidegree bound.
They also point out that the existence of a cycle of length $n-1$ often
helps to prove pancyclicity (say in \cite{HT1976, Overbeck1977}).
However, a potential difficulty to prove pancyclicity lays in the following fact.
The condition in Theorem \ref{thm:BGLhamilton} poses degree restriction on nonadjacent
dominated pairs only, which generalizes the class of LiSD.
Hence any result that the condition in Theorem \ref{thm:BGLhamilton} implies pancyclicity
must have the class of non-pancyclic LiSD as an exceptional class.
But a complete characterization of non-pancyclic LiSD's seems to be difficult and has not been found.

Therefore, it makes sense to firstly consider Bang-Jensen-Gutin-Li conditions that
involve both nonadjacent dominated pairs and nonadjacent dominating pairs, which generalize
LSD. So we conjecture the following.
Denote by $\mathcal{D}_L$ the class of non-pancyclic LSD's defined in Theorem \ref{thm:nonpancyclicLSD},
and by $\mathcal{D}_B$ the class of balanced complete bipartite digraphs with at least $4$ vertices.

\begin{conjecture} \label{conj:BGLpancyclic}
Let $D$ be a digraph with order $n\ge 3$. Suppose that for any $u\in V(D)$ belonging to a
nonadjacent dominated pair or a nonadjacent dominating pair, $d(u)\ge n$,
then $D$ is pancyclic, unless $D \in \mathcal{D}_L$ or $D \in \mathcal{D}_B$.
\end{conjecture}

The condition of the conjecture implies that of Theorem \ref{thm:BGLhamilton},
thus implies the existence of a hamiltonian cycle.
In this paper, we prove that this condition also implies
the existence of cycle of length $3$ and $n-1$, which supports the conjecture.
As mentioned above a cycle of length $n-1$ is often helpful to prove pancyclicity.
On the other hand, if we prove pacyclicity by induction on the cycle length,
then the existence of $3$-cycle is the basic step. So both results may be useful
to prove Conjecture \ref{conj:BGLpancyclic}.
The main results of the paper is the following two theorems, which are proved
in subsequent sections.

\begin{theorem} \label{thm:BGLTriangle}
Let $D$ be a strong digraph with order $n\ge 3$. Suppose that for every vertex $u$ which belongs to a nonadjacent
dominating pair or a nonadjacent dominated pair, $d(u)\ge n $, then $D$ has a $3$-cycle,
unless $n$ is even and $D \in \mathcal{D}_B$, or $B \in \mathcal{D}_L$ and every strong semicomplete digraph $S$ in the
round decomposition of $D$ satisfies $|S|\le 2$.
\end{theorem}

When $n=3$, by the definition of pancyclicity only $3$-cycle is required, thus in the next theorem
we assume $n\ge 4$.
\begin{theorem}\label{thm:BGL_N-1_Cycle}
Let $D$ be a strong digraph with order $n\ge 4$.
Suppose that for every vertex $u$ which belongs to a nonadjacent
dominating pair or a nonadjacent dominated pair, $d(u)\ge n $,
then $D$ has a cycle of length $n-1$, unless $n$ is even and $D \in \mathcal{D}_B$, or $D=C_n$
(which belongs to $\mathcal{D}_L$).
\end{theorem}

\section{Proof of Theorem \ref{thm:BGLTriangle}: the existence of a $\boldsymbol{3}$-cycle}

We need the following lemma to prove the theorem.
\begin{lemma}  \label{lem:DegreeSumLe2n}
Let $D$ be a strong digraph without directed $3$-cycle, and $u,v\in V(D)$.
If $uv, vu\in A(D)$, then $d(u)+d(v)\le 2n$, and equality holds iff for every
$x \in V(D)\backslash \{u,v\}$, $d(\{u,v\},x)=2$.
\end{lemma}
\begin{proof}
For every $x\in V(D)\backslash \{u,v\}$, we can not have $u\rightarrow x \rightarrow v$ or
$v\rightarrow x \rightarrow u$.
Therefore $|A(D)\cap \{ux, xv\}|\le 1$, $|A(D)\cap \{vx, xu\}|\le 1$ and thus
$d(\{u,v\},x)\le 2$, and so $d(u)+d(v)\le 2+2+2(n-2)=2n$.
If equality holds then $d(\{u,v\},x)= 2$ for every $x\in V(D)\backslash \{u,v\}$.
\end{proof}

Suppose that there is no $3$-cycle in $D$.
If both $u_0u_1\dots u_{k}$ and $u_{k}\dots u_1u_0$ are directed paths in $D$, we call
$u_0u_1\dots u_{k}$ a \emph{two-way path}, and a \emph{two-way arc} if $k=1$.
Now let's divide our discussion into two cases with respect to the longest two-way path in $D$.

\noindent Case 1. The longest two-way path in $D$ is of length at most $1$.

If there is no two-way path at all, then every vertex is of degree at most $n-1$.
By the condition of the theorem, there can not be any nonadjacent dominating pair or nonadjacent dominated pair,
and hence $D$ is an LSD.

If there is at least one two-way path of length $1$, say $uv$ in $D$, then we can prove that $u$ and $u$
have the same in-neighborhood and out-neighborhood in $D\backslash\{u,v\}$.
For otherwise, without lose of generality assume that there exists a vertex
$x$ such that $u\rightarrow x$ but $v \nrightarrow x$, since $D$ is $3$-cycle-free, $x \nrightarrow v$.
But then $\{x,v\}$ is a nonadjacent dominated pair, and by the condition of the theorem $d(v)\ge n$.
By Lemma \ref{lem:DegreeSumLe2n}, $d(u)+d(v)\le 2n$. Thus $d(v)\ge n \ge d(u)$.
Let $w$ be any vertex that is adjacent to both $u$ and $v$, we have that $d(u,w)=d(v,w)=1$.
Since $x$ is adjacent to $u$ but not $v$ and $d(v)\ge d(u)$,
there must exist another vertex $y$ which is adjacent to $v$ but not $u$.
But then $\{u,y\}$ forms a nonadjacent dominating pair or a nonadjacent dominated pair.
Hence, $d(u)\ge n$ and thus we have $d(u)=d(v)=n$.
Now equality in Lemma \ref{lem:DegreeSumLe2n} holds, and we have $d(\{u,v\},x)=2$.
Then we have $x\rightarrow u$ and $xuv$ is a two-way path of length $2$, a contradiction.
Thus, $u$ and $v$ have the same in-neighborhood and out-neighborhood in $D\backslash\{u,v\}$.

Now we can let $N^+_{uv}$ and $N^-_{uv}$ be the out-neighborhood and in-neighborhood
of both $u$ and $v$ in $D\backslash \{u,v\}$. Since there is no $3$-cycle, there can not be any arc
from $N^+_{uv}$ to $N^-_{uv}$.
By the strongness of $D$, there exists a vertex $x \notin \{u,v\} \cup N^+_{uv} \cup N^-_{uv}$
such that there is a path from $N^+_{uv}$ to $x$ and a path from $x$ to $N^-_{uv}$. Then
$$d(u)=d(v)=|N^+_{uv}|+|N^-_{uv}|+2\le n-3+2=n-1.$$

For any vertex $w$ which is not associate with a two-way arc, we also have $d(w)\le n-1$.
Hence $\delta(D)\le n-1$. By the condition of the theorem there can not be any nonadjacent dominating pair
and nonadjacent dominated pair in $D$. Thus $D$ is an LSD.

Now it has been shown that $D$ is an LSD in all cases. Since $D$ does not contain any $3$-cycle,
$D$ is not pancyclic and hence $D\in \mathcal{D}_L$.
Furthermore, every strong semicomplete digraph $S$ is pancyclic, i.e.,
contains cycle of every length from $3$ to $|S|$ (Theorem 1.5.3 of \cite{BJG2008}).
Hence, for every semicomplete digraph in the round decomposition of $D$,
we must have $|S|\le 2$.
$\newline$

\noindent Case 2. The longest two-way path in $D$ is of length at least $2$.

Firstly we claim that the longest two-way path is actually of length at least $3$ in this case.
Let $uvw$ be a two-way path. Since $D$ is $3$-cycle-free, $u$ and $w$ can not be adjacent.
Thus $\{u,w\}$ is a nonadjacent dominating pair, and $d(u),d(w)\ge n$.
But then $d(u, V(D)\backslash\{u,v,w\})\ge n-2$, and there is at least one vertex $x \in V(D)\backslash\{u,v,w\}$
which is adjacent to $u$ with a two-way arc. So $xuvw$ is a two-way path of length $3$.

Let $P=u_0u_1\dots u_{k}$ be a longest two-way path in $D$ with $k\ge 3$.
Since there is not $3$-cycle, for $0\le i \le k-2$, $u_i$ can not be adjacent to $u_{i+2}$ and
$\{u_i, u_{i+2}\}$ forms a nonadjacent dominated pair.
Hence $d(u_i)\ge n$ and $u_{i+2}\ge n$.
When $i$ runs from $0$ to $k-2$, we actually have $d(u_j)\ge n$ for $0\le j \le k$.
Together with Lemma \ref{lem:DegreeSumLe2n}, we have that $2n\le d(u_i)+d(u_{i+1})\le 2n$ for $0 \le i \le k-1$,
thus $d(u_j)=n$ for $0\le j \le k$.

Furthermore, for $0\le i \le k-1$, by Lemma \ref{lem:DegreeSumLe2n},
$d(\{u_i,u_{i+1}\}, x)=2$ for all $x\in V\backslash \{u_i,u_{i+1}\}$, and since
there is no $3$-cycle, we must have $A(D)\cap \{u_ix, xu_{i+1}\} = A(D)\cap \{u_{i+1}x, xu_{i}\} = 1$.
Thus, for $0\le j \le k-2$ and $y\in V\backslash\{u_j, u_{j+1}, u_{j+2}\}$,
if $u_j \rightarrow y$,
then $y \not \rightarrow u_{j+1}$ and $u_{j+2}\rightarrow y$,
and vice versa. Also note that $u_j$ and $u_{j+2}$ are nonadjacent.
Therefore,
$u_{j+2}$ and $u_j$ have the same out-neighborhood, that is
\begin{equation} \label{eqn:OutdegreeEqualforTwo}
N^+(u_j)=N^+(u_{j+2}).
\end{equation}
Let $j$ run through $0$ to $k-2$, we have
\begin{equation} \label{eqn:OutdegreeEqual}
N^+(u_0)=N^+(u_2)=\dots = N^+(u_{2\lfloor k/2 \rfloor}) \mbox{ and }
N^+(u_1)=N^+(u_3)=\dots = N^+(u_{2\lfloor (k-1)/2 \rfloor+1}).
\end{equation}
Similar equalities hold for in-degree of the vertices on $P$,
\begin{equation} \label{eqn:IndegreeEqual}
N^-(u_0)=N^-(u_2)=\dots = N^-(u_{2\lfloor k/2 \rfloor}) \mbox{ and }
N^-(u_1)=N^-(u_3)=\dots = N^-(u_{2\lfloor (k-1)/2 \rfloor+1}).
\end{equation}

Since $u_j$ and $u_{j+2}$ are not adjacent for any $0\le j \le k-2$,
 and $u_i$ and $u_{i+1}$ is connected by a two-way arc for any $0\le i \le k-1$,
(\ref{eqn:OutdegreeEqual}) and (\ref{eqn:IndegreeEqual}) imply that
the vertex subset $N_0=\{u_0, u_2, \dots, u_{2\lfloor k/2 \rfloor}\}$ is independent,
the vertex subset $N_1=\{u_1, u_3, \dots, u_{2\lfloor (k-1)/2 \rfloor+1}\}$ is independent as well,
and that all vertex in $N_0$ is adjacent to all vertex in $N_1$ by two-way arcs, or equivalently,
$\langle V(P) \rangle$ is a complete bipartite digraph.

Let $x\in V(D)\backslash V(P)$. Since $\langle V(P) \rangle$ is complete bipartite,
$x$ can not be adjacent to any vertices in $P$ through a two-way arc,
or we find a two-way path containing all vertices in $V(P)\cup \{x\}$, contradicting that $P$ is the longest.
However, for any $u_{2i}, u_{2j+1} \in V(P)$, $u_{2i}$ and $u_{2j+1}$ is connected by a two-way arc,
and $d(u_{2i})+d(u_{2i+1})=n+n=2n$, hence by Lemma \ref{lem:DegreeSumLe2n}, $d(\{u_{2i}, u_{2j+1}\}, x)=2$.
So $d(u_{2i},x)=d(u_{2j+1},x)=1$, and since there is no $3$-cycle,
either $x\rightarrow \{u_{2i},u_{2j+1}\}$ or $\{u_{2i},u_{2j+1}\}\rightarrow x$.
Since $u_{2i}$ and $u_{2j+1}$ are arbitrarily chosen, we have either
$x\rightarrow V(P)$ or $V(P)\rightarrow x$.

Let $N_P^+$ and $N_P^-$ be the set of $x\in V(D)\backslash V(P)$ such that
$V(P)\rightarrow x$ and $x\rightarrow V(P)$, respectively.
Then $N_P^+\cup N_P^-=V\backslash V(P)$. Since there is no $3$-cycle,
there can not be any arc from $N_P^+$ to $N_P^-$. Since $V(P)\cup N_P^+ \cup N_P^- = V(D)$,
$D$ can not be strong unless $N_P^+=N_P^-=\emptyset$.
Thus $V(D)=V(P)$, and $D$ is a complete bipartite digraph, furthermore since each vertex is of degree $n$,
$D$ must be a balance complete bipartite digraph, which completes the proof of Case 2 and the whole theorem.

\section{Proof of Theorem \ref{thm:BGL_N-1_Cycle}: the existence of an $\boldsymbol{(n-1)}$-cycle}\label{sec:Proof_N-1}
In the proof of this section, we will need several preliminary results. The first one is the
multi-insertion technique, which is a powerful tool for tackling paths and cycles problems in graphs.
A detail explanation and application of multi-insertion technique
in digraphs can be found in Chapter 6.4.2 of \cite{BJG2008}.
To maintain the integrity of the paper we briefly introduce the related concepts
and ideas here.

Let $P = u_0 u_2 \dots u_s$ be a path in a digraph $D$ and
let $Q = v_0 v_2 \dots v_t$ be a path in $D\backslash V(P)$.
The path $P$ is said to be able to be \emph{inserted into} $Q$ if there is an integer
$0\le i \le t-1$ such that $v_i \rightarrow u_1$ and $u_s \rightarrow v_{i+1}$,
in the sense that the path $Q$ can be extended to a new $(v_0 ,v_t)$-path
$Q[v_1 ,v_i ]PQ[v_{i+1} ,v_t ]$ with vertex set $V(P)\cup V(Q)$. Next, the path $P$ can
be \emph{multi-inserted into} $Q$ if there are integers $i_1 = 0 < i_2 < \dots < i_m = s+1$
 such that, for every $k \in \{1,2,\dots,m\}$, the subpath $P[u_{i_{k-1}} ,u_{i_k-1}]$ can be
inserted into $Q$. For the case that $Q$ is a cycle similar definitions can be given.

The definition of multi-insertion means that $P$ can be partition into subpaths each
of which can be inserted into $Q$.
The key result here is that under this assumption,
all vertices of $P$ can then be inserted into $Q$ to form a path
with vertex set $V(P)\cup V(Q)$ and the same endvertices with $Q$ (or merely a cycle
with vertex set $V(P)\cup V(Q)$, when $Q$ is a cycle).

\begin{lemma} \label{lem:Multi-Insertion} (The Multi-Insertion Technique)
Let $P$ be a path in a digraph $D$ and let $Q = v_0 v_1 \dots v_t$ be a path (a
cycle, respectively) in $D\backslash V(P)$. If $P$ can be multi-inserted into $Q$, then there
is a $(v_0 ,v_t)$-path $R$ (a cycle, respectively) in $D$ so that $V (R) = V (P)\cup V (Q)$.
\end{lemma}

The next lemma is two straightforward facts that frequently used in various literatures. We state
it without a proof.
\begin{lemma} \label{lem:VertexInsertToPath}
Let $D$ be a digraph.
\begin{enumerate}[label=(\alph*)]
\item Let $P=u_0u_1\dots u_{p-1}$ be a path with $|P| \ge 2$ in $D$ and $v\in V(D)\backslash V(P)$.
If $v$ can not be inserted into $P$ then $d_P(v)\le |P|+1$. If we further have
$v\nrightarrow u_0$ or $u_{p-1}\nrightarrow v$, then $d_P(v)\le |P|$.
\item Let $Q$ be a cycle with $|Q| \ge 3$ in $D$ and $v\in V(D)\backslash V(Q)$.
If $v$ can not be inserted into $Q$ then $d_Q(v)\le |Q|$.
\end{enumerate}
\end{lemma}

\begin{lemma} \label{lem:n-2Cycle} (\cite{HT1976}, Lemma 2)
Let $D$ be a strong digraph containing a cycle $C=u_0u_1\dots u_{n-3}u_0$ of length $n-2$,
where $n = |D|$. Let $v_0$, $v_1$ be the vertices not contained in $C$.
\begin{enumerate} [label=(\alph*)]
\item If $d(v_0), d(v_1)\ge n$ and $D$ does not contain a cycle of length $n-1$,
then $n$ is even and the notation may be chosen such that $v_0$ dominates and is dominated by
precisely $u_1,u_3\dots, u_{n-3}$,
and $v_1$ dominates and is dominated by precisely $u_0,u_2,\dots, u_{n-4}$.

\item If $D$ satisfies that $\delta(D)\ge n$ and contains no cycle of length $n-1$,
then $n$ is even and $D$ is isomorphic to $\overleftrightarrow{K}_{n/2,n/2}$.
\end{enumerate}
\end{lemma}

The below lemma shows that in a digraph satisfying the condition we raise,
a vertex adjacent to a longest nonhamiltonian cycle satisfies certain nice structure or degree conditions.

\begin{lemma} \label{lem:wToC}
Let $D$ be a strong digraph with order $n \ge 4$. Suppose that for every vertex $u$ which
belongs to a nonadjacent dominating pair or a nonadjacent dominated pair, $d(u) \ge n$.
Let $C$ be a longest nonhamiltonian cycle in $D$ with $|C| \le n-2$, and $w\in V(D)\backslash V(C)$
be adjacent to some vertex on $C$. If $d(w)\le n-1$, then either $w\rightarrow C$ or $C\rightarrow w$.
\end{lemma}

\begin{proof}(of Lemma \ref{lem:wToC})
Let $C=u_0u_1\dots u_{k-1}$. By the condition of the lemma $w$ is adjacent to at least one vertex on $C$,
say $u_0 \rightarrow w$. Since $C$ is the longest nonhamiltonian cycle with $|C|\le n-2$, $w$ can not
be inserted into $C$, hence $w\nrightarrow u_1$. But $w$ and $u_1$ are dominated by $u_0$, and by the
condition of the lemma and $d(w)\le n-1$, they must be adjacent. So $u_1\rightarrow w$. Repeatedly
apply the above arguments we conclude that $u_i \rightarrow w$, for $0\le i \le k-1$, that is,
$C\rightarrow w$. Similarly if we assume that $w\rightarrow u_0$ at the beginning we will get
$w\rightarrow C$.
\end{proof}

Now we are ready to prove Theorem \ref{thm:BGL_N-1_Cycle}. We firstly prove the theorem for $n\in \{4,5\}$.
By above discussion there is a hamiltonian cycle in $D$.

If $n=4$ and $D$ is not isomorphic to $C_4$ (note that $C_4\in \mathcal{D}_B\cap \mathcal{D}_L$),
then $D$ has a chord, which will produce a $3$-cycle, as request.

If $n=5$ and and $D$ is not isomorphic to $C_5$.
We denote the hamiltonian cycle in $D$ by $C=u_0u_1u_2u_3u_4u_0$,
then $C$ has at least one chord.
Assume by contradiction that $D$ does not has a $4$-cycle,
then every chord $u_iu_{i+2}$ of $C$ does not exist,
or we have a $4$-cycle $u_iu_{i+2}\dots u_{i}$, where the
subscript are taken modulo $n$. Thus there can only be chords of the form $u_iu_{i-2}$.
(1) Now suppose that $C$ has only one chord, say $u_2u_0$,
then $u_0$ and $u_3$ form a nonadjacent dominated pair, hence $d(u_3)\ge 5$.
But there is no chord with end $u_3$, thus $d(u_3)\le 4$, a contradiction.
Therefore $C$ has at least two chords. (2) If there are two chords associate with one vertex,
say we have the chords $u_2u_0$ and $u_0u_3$, then we have a path $u_1u_2u_0u_3u_4$,
and thus we can not have the chords $u_3u_1$ and $u_4u_2$, or a $4$-cycle can be formed.
But then $\{u_1,u_3\}$ and $\{u_2,u_4\}$ are nonadjacent dominated or dominating pairs.
Thus $d(u_i) \ge 5$ for $i\in \{1,2,3,4\}$. Then the reverse arc of every arc on $C$ must exist,
and so as $u_1u_4$. Now there can be quite a few $4$-cycles in $D$, say $u_4u_3u_2u_1u_4$,
again a contradiction. (3) If there are at least three chords, we must have two chords with
the same end which has been discussed in (2).
(4) So there are exactly two chords which are without a common endvertex, say $u_0u_2$
and $u_3u_1$, then $u_4$ and $u_1$ forms a nonadjacent dominated pair, and $d(u_4)\ge 5$,
but $u_4$ is not associate with any chord and hence $d(u_4)\le 4$, a contradiction,
which completes the proof for $n=5$.
\newline

Now we assume $n\ge 6$ and there is no $(n-1)$-cycle in $D$.
We further assume that $D$ is not isomorphic to $C_n$.
Note that the condition of the theorem imply a hamiltonian cycle.
Since $D$ is not isomorphic to $C_n$, there is a chord on the hamiltonian cycle which produces
a shorter cycle. Thus we can a find nonhamiltonian cycle in $D$.
Let $C=u_0u_1\dots u_{k-1}u_0$ be the longest nonhamiltonian cycle in $D$, where $3\le k \le n-2$.

We claim that there exists a $C$-bypass in $D$.
Suppose to the contrary that there is no $C$-bypass.
By the strongness of $D$ there exists a cycle $Q$ with $|V(Q)\cap V(C)|=1$.
W.l.o.g. we may further assume the common vertex of $Q$ and $C$ to be $u_0$
and denote the successor of $u_0$ on $Q$ by $v$.
Now any arc between $v$ and $u_\beta$ with $\beta \ne 0$ will produce a $C$-bypass and
leads to a contradiction, hence $d_C(v) \le 2$. Meanwhile, $\{v, u_{1}\}$
is a nonadjacent dominated pair and thus $d(v), d(u_{1})\ge n$ by the condition
of the theorem.
For any $w\in V(R)\backslash v$, if $v\rightarrow w \rightarrow u_1$ or
$u_1\rightarrow w \rightarrow v$ we have a path between $v$ and $ u_1$,
which again forms a $C$-bypass with $Q[u_0,v]$ or $Q[v,u_0]$, a contradiction.
Therefore $d(w,\{v, u_1\})\le 2$, and in total
$d(V(R)\backslash v,\{v, u_{1}\})\le 2(|R|-1)$.
Then
$$2n\le d(v)+d(u_{1})= d_C(v)+d_C(u_{1})+d(V(R)\backslash v,\{v, u_{1}\}))
\le 2+2(|C|-1)+2(|R|-1)=2n-2,$$
a contradiction. Hence a $C$-bypass must exists.
\newline

We take a $C$-bypass $P=x_0x_1\dots x_{l}$ with the minimum gap, and
w.l.o.g. we assume  $x_0=u_0$ and let $x_{l}=u_\alpha$, where $l \ge 2$ is the
length of the bypass, and $\alpha\ge 1$ is the length of the gap.
Further, let $C'=C[u_1, u_{\alpha-1}]$ and $C''=C[u_\alpha, u_{0}]$.
We divide our discussion into the cases that $|C|\le n-3$ and $|C|=n-2$.
\newline

\textbf{Case 1}. $|C|\le n-3$, then $|R|\ge 3$.
We further divide this case into two, according to whether we can
choose $P$ so that $V(R)\backslash V(P) \ne \emptyset$. $\newline$

\textbf{Case 1.1}. There exists a bypass $P$ with minimum gap and $|V(R)\backslash V(P)| \ge 1$.

In this case we have $\alpha \ge 2$, for if $\alpha=1$ then $P$ can be inserted into $C$ to
form a longer nonhamiltonian cycle (because $|V(R)\backslash V(P)| \ge 1$), contradicting the choice of $C$.
And by $\alpha \ge 2$, $C'$ contains at least one vertex $u_1$.

Since $P$ has the minimum gap with respect to $C$,
$x_1$ can not be adjacent to any vertex $u_\beta$ on $C'$,
and there is no vertex $w$ in $V(R)\backslash\{x_1\}$
such that $x_1 \rightarrow w \rightarrow u_\beta$ or $u_\beta \rightarrow w \rightarrow x_1$,
therefore we have
\begin{equation} \label{eqn:sumonC'}
d_{C'}(x_1)+d_{C'}(u_\beta) \le 2(|C'|-1),
\end{equation}

and
\begin{equation} \label{eqn:sumonR}
d_{R}(x_1)+d_{R}(u_\beta) \le 2(|R|-1)=2(n-k-1),
\end{equation}

Since $C$ is the longest nonhamiltonian cycle with $|C| \le n-3$,
$x_1$ can not be inserted into $C''$, by Lemma \ref{lem:VertexInsertToPath} (a),
\begin{equation} \label{eqn:x_1degreeonC''}
d_{C''}(x_1)\le |C''|+1
\end{equation}

Now we show that for any $u_\gamma$ on $C'$ such that $u_0\rightarrow u_\gamma$
(there is at least one such $u_\gamma$, that is, $u_1$),
\begin{equation} \label{eqn:degreeonC''}
d_{C''}(u_\gamma)\ge |C''|+3.
\end{equation}

The fact that $x_1$ and $u_\gamma$ are nonadjacent and both dominated by $u_0$ implies
that $d(x_1), d(u_\gamma)\ge n$. Substituting $\beta$ for $\gamma$ in (\ref{eqn:sumonC'}) and (\ref{eqn:sumonR}),
and together with (\ref{eqn:x_1degreeonC''}), we have
\begin{equation} \label{eqn:sum}
\begin{split}
2n &\le d(x_1)+d(u_\gamma) \\
&= d_{R}(x_1)+d_{R}(u_\gamma)+d_{C'}(x_1)+d_{C'}(u_\gamma)+ d_{C''}(x_1)+d_{C''}(u_\gamma) \\
&\le 2n-|C''|-3+d_{C''}(u_\gamma)
\end{split}
\end{equation}
which implies (\ref{eqn:degreeonC''}).

Note that $u_0\rightarrow u_1$ hence $d_{C''}(u_1)\ge |C''|+3$, thus $u_1$ can be inserted into $C''$.
If $C'$ can be multi-inserted into $C''$ then by Lemma \ref{lem:Multi-Insertion} we obtain a cycle $Q$
with $V(Q)=V(C)\cup V(P)$,
Since $|V(R)\backslash V(P)| \ge 1$, $Q$ is nonhamiltonian but longer than $C$,
contradicting the assumption we have made on $C$.
Therefore $C'$ can not be multi-inserted into $C''$.

Since $u_1$ can be inserted into $C''$, there must exist an $\eta \in \{2, \dots, \alpha-1 \}$
such that $C[u_1, u_{\eta-1}]$ can be multi-inserted into $C''$, but $C[u_1, u_{\eta}]$ cannot.
In particular, by Lemma \ref{lem:Multi-Insertion},
$u_\eta$ cannot be inserted into $C''$. Then by (\ref{eqn:degreeonC''}), $u_0 \nrightarrow u_\eta$,
and by Lemma \ref{lem:VertexInsertToPath} (a), $d_{C''}(u_\eta)\le |C''|$.
Together with (\ref{eqn:sumonC'}), (\ref{eqn:sumonR}) and (\ref{eqn:x_1degreeonC''})
we have $d(x_1)+d(u_\eta) \le 2n-3$, but we already have $x_1\ge n$, so
\begin{equation} \label{eqn:degreeofu_eta}
d(u_\eta)\le n-3.
\end{equation}

By the definition of multi-insertion, there exist $u_j$ on $C[u_1,u_{\eta-1}]$
and $u_i$ on $C[u_\alpha, u_{k-1}]$ such that $u_i \rightarrow u_j$ and $u_{\eta-1} \rightarrow u_{i+1}$.
Now $u_\eta$ and $u_{i+1}$ are both dominated by $u_{\eta-1}$, by (\ref{eqn:degreeofu_eta})
and the condition of the theorem, they must be adjacent.
If $u_\eta \rightarrow u_{i+1}$, then $C[u_1 ,u_\eta ]$ can be inserted into $C''$, contradicting our assumption.
Hence $u_{i+1}\rightarrow u_\eta$.
Considering $u_\eta$ and $u_{i+2}$, by analog argument we conclude that $u_{i+2}\rightarrow u_\eta$.
Continuing this process, we finally conclude that $u_0\rightarrow u_\eta$, contradicting the above
argument that $u_0 \nrightarrow u_\eta$, and complete the proof of Case 1.1.
$\newline$

\textbf{Case 1.2}. For every bypass $P$ with minimum gap, $V (R)\backslash V (P)=\emptyset$.

In this case, it is possible that $\alpha =1$ which means $C'$ is empty,
for the insertion of $P$ into C results in a hamiltonian cycle, which will not
cause any contradiction like that in Case 1.1.
Now we firstly assume $\alpha \ge 2$, then  $|C'|\ge 1$ and $|C''|\le k-1$.

Let's estimate the degree of $x_1$.
Since $P$ has the minimum gap $x_1$ is not adjacent to any vertex on $C'$.
If $x_1 \rightarrow x_i$ with $3\le i \le l$, then we
have a gap $P'=x_0x_1P[x_i,x_l]$ with minimum gap and $|V(R)\backslash V(P')|\ge 1$,
contradicting the assumption of this case. Therefore $x_1 \nrightarrow x_i$ for $3\le i \le l$.
Hence
\begin{equation} \label{eqn:degree_x1_in_R}
d_R(x_1)\le 2+|R|-2=|R|=n-k.
\end{equation}
Since $C$ is the longest nonhamiltonian cycle and $|C|\le n-3$, $x_1$ can not be inserted into $C$
and particulary $x_1$ can not be inserted into $C''=C[u_\alpha, u_0]$.
Note that $x_1 \nrightarrow x_l=u_\alpha$, by Lemma \ref{lem:VertexInsertToPath} (a),
\begin{equation} \label{eqn:degree_x1_in_C''}
d_{C''}(x_1)\le |C''|.
\end{equation}
By (\ref{eqn:degree_x1_in_R}) and (\ref{eqn:degree_x1_in_C''}), we have the following contradiction.
\begin{equation} \label{eqn:degree_x1}
n\le d(x_1)=d_R(x_1)+d_{C''}(x_1)+d_{C'}(x_1)\le n-k+|C''|+0\le n-k+k-1=n-1.
\end{equation}

Now we must have $\alpha = 1$, and hence $x_l=u_1$.

Consider the degree of $x_1$. $x_1$ is dominated by $x_0=u_0$ which
is on $C$, however $x_1$ can not be dominated by $u_{k-1}$, or we will have a $(n-1)$-cycle
$P[x_1,u_1]C[u_1,u_{k-1}]x_1$. Therefore neither $x_1 \rightarrow C$ nor $C \rightarrow x_1$.
By Lemma \ref{lem:wToC}, $d(x_1)\ge n$.

Since $x_1$ can not be inserted into $C$, by Lemma \ref{lem:VertexInsertToPath} (b),
$d_C(x_1)\le |C|$. If $x_1 \rightarrow x_i$ for any $3\le i \le l$,
we will have a bypass with the minimum gap but does not contain all vertices in $R$, contradicting the
assumption of the case. Therefore $d_R(x_1)\le |R|-1+1=|R|$. So $n\le d(x_1)=d_C(x_1)+d_R(x_1) \le |C|+|R|=n$
and all equalities must hold, in particular $d_C(x_1)=|C|$.

Since $u_0\rightarrow x_1$, $u_{k-1}\nrightarrow x_1$ and $d_C(x_1)=|C|$,
the exist some arcs from $x_1$ to $C$. Further since $x_1$ can not be inserted into $C$,
$x_1\nrightarrow u_1$. Let $j$ be that $x_1\rightarrow u_j$
but $x_1 \nrightarrow u_i$ for all $1\le i \le j-1$. Then $j\ge 2$.
Since $x_1$ can not be inserted into $C$, $u_{j-1}\nrightarrow x_1$ and $u_{j-1}$ is not
adjacent to $x_1$. Now let $0\le r \le j-1$ be that $u_r \rightarrow x_1$ but
$u_i$ is not adjacent to $x_1$ for all $r+1\le i \le j-1$.
Then $u_rx_1u_j$ is a bypass with a gap of length at least $2$, where $x_1$ is not adjacent
to the internal vertices of the gap.

If $j=r+2$, consider the degree sum of $x_1$ and $u_{r+1}$.
They are dominated by $u_r$ and nonadjacent, therefore $d(x_1), d(u_{r+1})\ge n$.
Since $|R|\ge 3$, for every $w\in R\backslash x_1$,
we can not have $x_1\rightarrow w \rightarrow u_{r+1}$ or $u_{r+1} \rightarrow w \rightarrow x_1$,
or we have a bypass with a smaller gap, contradicting the choice of $P$.
Hence $d_R(x_1)+d_R(u_{r+1})\le 2(|R|-1)$.
Note that $C[u_{r+2},u_r]x_1u_{r+2}$ is another longest nonhamiltonian cycle, therefore $u_{r+1}$ can not
be inserted into this cycle, in particular $u_{r+1}$ can not be inserted into the path $Q=C[u_{r+2},u_r]$.
By Lemma \ref{lem:VertexInsertToPath} (a), $d_C(u_{r+1})=d_{Q}(u_{r+1})\le |Q|+1=k$.
Together with $d_C(x_1)=|C|$, we have
$d(x_1)+d(u_{r+1})\le d_R(x_1)+d_R(u_{r+1})+d_C(x_1)+d_C(u_{r+1})\le 2(|R|-1)+2k=2n-2$, contradicting
$d(x_1), d(u_{r+1})\ge n$.

Therefore $j\ge r+3$. However, in this case we will have $d(x_1)\le n-1$ which is a contradiction.
Let $Q=C[u_j,u_{r}]$, $x_1$ can not be inserted into $Q$, hence by Lemma \ref{lem:VertexInsertToPath} (a),
$d_C(x_{1})=d_{Q}(x_1)\le |Q|+1=k-(j-r-1)+1\le k-1$. So $d(x_1)=d_C(x_{1})+d_R(x_{1}) \le |R|+k-1 = n-1$,
a contradiction that complete the proof of Case 1.2.

$\newline$

\textbf{Case 2}. $|C|=n-2$, and hence $|R|=2$.

In this case we apply Lemma \ref{lem:n-2Cycle}. Now that $C=u_0u_1\dots u_{n-3}u_0$ and we denote
the vertices not in $C$ by $v_0$ and $v_1$.

We prove that $d(v_0)\ge n$ and $d(v_1)\ge n$. Suppose to the contrary that at least one of them does not hold,
say $d(v_0)\le n-1$.

If $v_0$ is not adjacent to any vertex on $C$, then by the strongness of $D$, $v_0 \rightarrow v_1$, $v_1\rightarrow v_0$
and $v_1$ dominate and is dominated by some vertex on $C$. But then $v_0$ is contained in a nonadjacent dominated pair
and thus $d(v_0)\ge n$, a contradiction to the assumption $d(v_0)\le n-1$.
Hence, $v_0$ is adjacent to some vertex on $C$.

By Lemma \ref{lem:wToC}, $v_0\rightarrow C$ or $C\rightarrow v_0$. W.l.o.g assume that $v_0 \rightarrow C$.
By the strongness of $D$, there must exist a $0\le j \le k-3$ such that $u_j \rightarrow v_1 \rightarrow v_0$.
But then we have an $(n-1)$-cycle $v_0u_{j-2}Cu_jv_1v_0$, a contradiction.
Hence $d(v_0)\ge n$, and similarly $d(v_1)\ge n$.

Now $C$, $v_0$ and $v_1$ satisfy the condition of Lemma \ref{lem:n-2Cycle} (a), by the conclusion
of (a), $n$ is even and we may assume $v_0$ dominates and is dominated by precisely
$u_1 ,u_3,\dots ,u_{n-3}$, and $v_1$ dominates and is dominated by precisely $u_0 ,u_2 ,\dots,u_{n-4}$.
Then $u_2$ and $v_0$ are nonadjacent and both dominated by $u_1$, hence $d(u_2)\ge n$, and similarly
$d(u_i)\ge n$ for $0\le i \le n-3$. So we have $\delta(D)\ge n$ and the condition of Lemma \ref{lem:n-2Cycle} (b)
is satisfied, then we conclude that $D$ is isomorphic to $\overleftrightarrow{K}_{n/2,n/2}$, thus
complete the proof of Case 2 and the theorem.

\end{document}